\pdfoutput=1
\documentclass[11pt]{extarticle}

\usepackage[left=1in,right=1in,top=1in,bottom=1.3in]{geometry}
\usepackage{amsmath, amssymb, amsthm, latexsym, indentfirst, amsfonts, xcolor, mathtools, manfnt, microtype, soul, subcaption, booktabs}
\usepackage[utf8]{inputenc} 
\usepackage[english]{babel}

\usepackage[breaklinks,backref=page]{hyperref}
\usepackage{cleveref}
\hypersetup{
	colorlinks = true, % Colours links instead of ugly boxes
	urlcolor = cyan, % Colour for external hyperlinks
	linkcolor = teal, % Colour of internal links
	citecolor = cyan % Colour of citations
}

\usepackage{tikz} % for figures

\renewcommand*{\backref}[1]{}
\renewcommand*{\backrefalt}[4]{
	\ifcase #1 \textcolor{red}{Not cited.}%
	\or $\uparrow$#2%
	\else $\uparrow$#2%
	\fi%
}
\let\OLDthebibliography\thebibliography
\renewcommand\thebibliography[1]{
	\OLDthebibliography{#1}
	\setlength{\parskip}{0pt}
	\setlength{\itemsep}{0pt plus 0.3ex}
}

\newtheorem{Theorem}{Theorem}

\newtheorem{Corollary}{Corollary}

\newtheorem{Construction}{Construction}

\newtheorem{Lemma}{Lemma}

\newtheorem{Question}{Question}
\newtheorem{Proposition}{Proposition}

\newcommand{\sm}{\!\setminus\!}
\newcommand{\R}{\mathbb R}
\newcommand{\Z}{\mathbb Z}
\newcommand{\N}{\mathbb N}

\newcommand{\ex}{{\rm ex}}

\newcommand{\x}{\mathbf x}
\newcommand{\y}{\mathbf y}
\newcommand{\z}{\mathbf z}
\newcommand{\aaa}{\mathbf a}
\newcommand{\cH}{\mathcal H}
\newcommand{\cF}{\mathcal F}

\title{Vertex-Ramsey theorems for Cartesian powers of graphs}

\author{N\'ora Alm\'asi\thanks{Budapest University of Technology and Economics, Alfr\'ed R\'enyi Institute of Mathematics, Budapest, Hungary, \href{mailto:almasi.nora@cs.bme.hu}{\tt almasi.nora@cs.bme.hu}.}
\and Maria Axenovich\thanks{Karlsruhe Institute of Technology, Karlsruhe, Germany, \href{mailto:maria.aksenovich@kit.edu}{\tt maria.aksenovich@kit.edu}.}
\and Arsenii Sagdeev\thanks{Karlsruhe Institute of Technology, Karlsruhe, Germany, \href{mailto:sagdeevarsenii@gmail.com}{\tt sagdeevarsenii@gmail.com}.}}
\date{}

\begin{document}

\maketitle

\begin{abstract}
    For graphs $G,H$ and positive integers $r$ and $n$ we write $G^{\square n} \xrightarrow{r} H$  if every $r$-vertex-coloring of the Cartesian power $G^{\square n}$ of $G$ contains a monochromatic copy  of $H$. Since chromatic number $\chi$ of $G^{\square n}$ is the same as $\chi(G)$, there is an $r$-vertex coloring of $G^{\square n}$ for $r=\chi(G)$, such that each color class is an independent set. 
    We prove that for $r<\chi(G)$ there is a large class of graphs $H$ such that $G^{\square n} \xrightarrow{r} H$. These graphs are so-called layered graphs in a hypercube. We also show that for some graphs $G$, such as for example odd cycles or cliques,  the class of layered graphs $H$ is the only one satisfying the above Ramsey property when $\chi(G)/2 < r < \chi(G)$.
    In addition, we prove a more general result relating Ramsey properties of $G$ and graphs $H$ such that $G^{\square n} \xrightarrow{r} H$. One of the technical tools is a Ramsey-type statement for discrete cubes $[m]^n$ that we call the Cube Layered Lemma, which is of independent interest. One of the original motivations for studying Ramsey properties of Cartesian powers of $G$ is the fact that $G^{\square n}$ is a unit distance graph if $G$ is a unit distance graph. This provides applications in Euclidean Ramsey theory.
\end{abstract}

%%%%%%%%%%%%%%%%%%%%%%%%%%%%%%%%%
%%%%%%%%%%%%%%%%%%%%%%%%%%%%%%%%%
\section{Introduction}
%%%%%%%%%%%%%%%%%%%%%%%%%%%%%%%%%
%%%%%%%%%%%%%%%%%%%%%%%%%%%%%%%%%

Ramsey theory, originating from a work by Ramsey \cite{R}, studies unavoidable structures in the partition of discrete objects. In this paper, we consider vertex-Ramsey properties of Cartesian powers of graphs. Given two graphs $G$ and $H$, their \emph{Cartesian product $G\square H$} is defined as the graph on the vertex set $V(G)\times V(H)$, where the vertices $(x_1,x_2),(y_1,y_2)$ form an edge in $G\square H$ if and only if either $x_1y_1\in E(G)$ and $x_2=y_2$, or $x_1=y_1$ and $x_2y_2\in E(H)$. Let $G^{\square 1}=G$ and for $n\geq 2$, let $G^{\square n}= G^{\square n-1} \square G$ be the $n^{\mathrm{th}}$ {\it Cartesian power} of $G$. Note that $K_2^{\square n}$ is a binary {\it hypercube} that is sometimes denoted $Q_n$. Letting $V(K_2)=\{0,1\}$, an $i$th  {\it vertex-layer} of $Q_n$ is a set of all vertices with exactly $i$ ones. The $i$th {\it edge-layer} of $Q_n$ is its subgraph induced by the $i$th and $(i+1)$st vertex layers. A graph $G$ is called {\it layered} if it is a subgraph of an edge-layer of some hypercube. For properties of layered graphs, see for example \cite{AMW, AS, BLMW}.

For graphs $G,H$ and $r,n\in \N$ we write $G^{\square n} \xrightarrow{r} H$  if every $r$-coloring of the vertices of $G^{\square n}$ contains a monochromatic copy  of $H$. If $n=1$, we simply write $G \xrightarrow{r} H$. 
The chromatic number $\chi(H)$ of a graph $H$ is the smallest integer $r+1$ such that $H \xrightarrow{r} K_2$.
We consider graphs $G$ and $H$ such that for sufficiently large $n$ and an integer $r$,  $G^{\square n} \xrightarrow{r} H$. Our goal is to understand the relation between $G$, $H$, and $r$. Observe first that if $r\ge \chi(G)$, then $G^{\square n} \not \xrightarrow{r} K_2$ because $\chi(G)= \chi(G^{\square n})$ for any $n\ge 1$. So the problem is non-trivial only if $r<\chi(G)$. Our first result deals with the first non-trivial case $\chi(G)=r+1$.
 
 Our main result provides a large class of graphs $H$ such that  $G^{\square n} \xrightarrow{r} H$. 

\begin{Theorem} \label{arrow_layered}
    Let $H$ be a layered graph, $r$ be an integer,  and $G$ be a graph of chromatic number $r+1$. Then there exists $n \in \N$ such that $G^{\square n} \xrightarrow{r} H$.
\end{Theorem}

We further show that for graphs  $G$ such as a complete graph or an odd cycle, being layered is a characterization of graphs $H$ such that $G^{\square n} \xrightarrow{r} H$. Note that Cartesian powers of cycles are usually referred to as \textit{toroidal grids}, and this object is important for eliminating edge effects in simulations, optimizing parallel computer architectures, and modeling biological brain maps. The Cartesian powers of cliques are called {\it Hamming graphs}, with many applications in coding theory and algebraic combinatorics.

\begin{Proposition}\label{prop:necessary}
     Let $s, r,$ and $n$ be positive integers, such that $s/2<r<s$.
      If $C_{2s+1} ^{\square n} \xrightarrow{2} H$ or $K_s^{\square n} \xrightarrow{r} H$, then $H$ is layered.
 \end{Proposition}

Theorem \ref{arrow_layered} strengthens the work of Axenovich, Liu and Sagdeev \cite{ALS}, who proved via Turán-type arguments an analogous statement for graphs $H$ that have zero hypercube Turán density, as well as for the case when $G=K_3$ and $H$ is layered. 
A graph $H$ is of {\it zero hypercube Turán density} if $\ex (Q_n, H) /|E(Q_n)| \rightarrow 0$ as $n$ goes to infinity, where $\ex(Q_n, H)$ is the largest size of a subgraph of $Q_n$ not containing $H$ as a subgraph.
Every graph of zero hypercube Tur\'an density is layered. However, there is a large class of layered graphs that are not of zero hypercube Tur\'an density. For example, these include graphs containing cycles of length $6$ or $10$. \Cref{prop:necessary} answers Question~14 from~\cite{ALS} in a strong form.

Our second result shows that if $ \chi(G)>2r$, then in addition to layered graphs, sufficiently long odd cycles also become unavoidable. We write $G \xrightarrow{r} \{H_1, H_2, \ldots, H_m\}$ if any $r$-coloring of vertices of $G$ contains a monochromatic subgraph isomorphic to $H_i$, for some $i\in \{1, \ldots, m\}$.

\begin{Theorem} \label{arrow_cycles}
    Let $G$ be a graph on $m$ vertices and $r,\ell$ be integers such that $G \xrightarrow{r}\{C_3, C_5, \dots, C_{2\ell+1}\}$. Then for every $s\ge\ell+2$ there exists $n\in \N$ such that $G^{\square n} \xrightarrow{r} C_{2s+1}$. In particular, if $\chi(G)>2r$ and $2s+1 \ge m+4$ then there exists $n\in \N$ such that $G^{\square n} \xrightarrow{r} C_{2s+1}$.
\end{Theorem}

We shall prove  Theorems \ref{arrow_layered} and \ref{arrow_cycles} using a more general Theorem \ref{grapes} below. To state it, we need some definitions. 
Let $V_k$ denote the $k$th vertex layer of a hypercube consisting of binary vectors with exactly $k$ ones.
Let $L=L_k$ be the middle, $k$th layer of the hypercube $Q_{2k-1}$ induced by vertex layers $V_{k-1}$ and $V_k$. We represent each edge $e=vu$ of $Q_n$ as a vector coinciding with $u$ and $v$ in the $n-1$ positions where they are the same, and that has $*$ in the position (direction) where $u$ and $v$ differ.  Let $E^*=E^*(L)$ be a subset of the edge set $E(L)$ whose binary-star representation has no ones following the star, for example $1101\!*\!00$ or $101010*$. Observe that $E^*$ is constructed by taking each vertex from the top layer of $L$ and choosing exactly one edge incident to it, corresponding to the largest direction. In particular, $E^*$ forms a star forest whose leaves span the top vertex layer of $L$.
For a graph $F$ and a pair $(a,b)$ of some adjacent vertices of $F$, 
let $L_k(F, (a,b))$ be a graph obtained by
``gluing" copies of $F$ to $L$ by identifying the edges corresponding to $ab$ to respective edges of $E^*$, such that the vertices corresponding to $a$ in all these copies are in the upper vertex layer.
Formally, let $F_e$, $e\in E^*$ be copies of $F$ such that for each $e\in E^*$, the vertex of $F_e$ corresponding to $a$ is $e\cap V_{k}$ and the vertex of $F_e$ corresponding to $b$ is $e\cap V_{k-1}$. Moreover, no other vertex of $F_e$ appears in the vertex set of $L_k$ and in the vertex set of any  $F_{e'}$, $e'\in E^*\setminus \{e\}$.
Finally $L_k(F, a,b)=L_k \cup  \bigcup_{e\in E^*}F_e$, see Figure \ref{fig} for an illustration.

%%%%%%%%%%%%%%%%%%%%%%%%%%%%%%%%%%%%%%%%%%
%%%%%%%%%%%%%%%%%%%%%%%%%%%%%%%%%%%%%%%%%%
%%%%%%%%%%%%%%%%%%%%%%%%%%%%%%%%%%%%%%%%%%

% A narrow 5-cycle whose fifth edge is #2--#3.
% The optional argument may be "above" or "below".
\newcommand{\slimfivecycle}[4][above]{%
  \path (#2) -- (#3)
    node[pos=.22,sloped,coordinate,#1=1.2mm] (#4a) {}
    node[pos=.50,sloped,coordinate,#1=1.7mm] (#4b) {}
    node[pos=.78,sloped,coordinate,#1=1.2mm] (#4c) {};
  \fill[yellow!35] (#2)--(#4a)--(#4b)--(#4c)--(#3)--cycle;
  \draw[cycle edge] (#2)--(#4a)--(#4b)--(#4c)--(#3);
  \node[aux vertex] at (#4a) {};
  \node[aux vertex] at (#4b) {};
  \node[aux vertex] at (#4c) {};
}

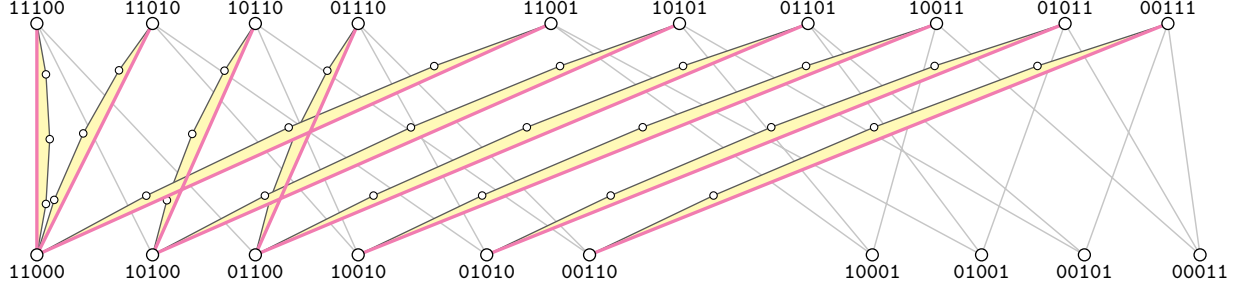
\begin{figure}
\begin{tikzpicture}[
scale=1.7,
  main vertex/.style={
    circle,
    draw,
    fill=white,
    inner sep=0pt,
    minimum size=4.5pt,
    line width=.5pt
  },
  aux vertex/.style={
    circle,
    draw,
    fill=white,
    inner sep=0pt,
    minimum size=2.8pt,
    line width=.4pt
  },
  incidence/.style={
    draw=gray!45,
    line width=.55pt
  },
  special/.style={
    draw=magenta!65,
    line width=1.25pt
  },
  cycle edge/.style={
    draw=black!65,
    line width=.5pt
  },
  bit label/.style={
    font=\ttfamily\scriptsize,
    inner sep=1pt
  }
]

%------------------------------------------------
% Vertex coordinates
%------------------------------------------------

% Upper row
\coordinate (t1)  at (0.00,1.80);
\coordinate (t2)  at (0.90,1.80);
\coordinate (t3)  at (1.70,1.80);
\coordinate (t4)  at (2.50,1.80);

\coordinate (t5)  at (4.00,1.80);
\coordinate (t6)  at (5.00,1.80);
\coordinate (t7)  at (6.00,1.80);
\coordinate (t8)  at (7.00,1.80);
\coordinate (t9)  at (8.00,1.80);
\coordinate (t10) at (8.80,1.80);

% Lower row
\coordinate (b1)  at (0.00,0);
\coordinate (b2)  at (0.90,0);
\coordinate (b3)  at (1.70,0);
\coordinate (b4)  at (2.50,0);
\coordinate (b5)  at (3.50,0);
\coordinate (b6)  at (4.30,0);

\coordinate (b7)  at (6.50,0);
\coordinate (b8)  at (7.35,0);
\coordinate (b9)  at (8.15,0);
\coordinate (b10) at (9.05,0);

%------------------------------------------------
% Gray incidence edges
%------------------------------------------------

\draw[incidence]
  (t1)--(b1) (t1)--(b2) (t1)--(b3)
  (t2)--(b1) (t2)--(b4) (t2)--(b5)
  (t3)--(b2) (t3)--(b4) (t3)--(b6)
  (t4)--(b3) (t4)--(b5) (t4)--(b6)

  (t5)--(b1) (t5)--(b7) (t5)--(b8)
  (t6)--(b2) (t6)--(b7) (t6)--(b9)
  (t7)--(b3) (t7)--(b8) (t7)--(b9)
  (t8)--(b4) (t8)--(b7) (t8)--(b10)
  (t9)--(b5) (t9)--(b8) (t9)--(b10)
  (t10)--(b6) (t10)--(b9) (t10)--(b10);

%------------------------------------------------
% Slim 5-cycles hanging on the magenta edges
%------------------------------------------------

\slimfivecycle{t1}{b1}{c1}
\slimfivecycle{t2}{b1}{c2}
\slimfivecycle{t3}{b2}{c3}
\slimfivecycle{t4}{b3}{c4}

\slimfivecycle{t5}{b1}{c5}
\slimfivecycle{t6}{b2}{c6}
\slimfivecycle{t7}{b3}{c7}
\slimfivecycle{t8}{b4}{c8}
\slimfivecycle{t9}{b5}{c9}
\slimfivecycle{t10}{b6}{c10}

%------------------------------------------------
% Magenta edges
%------------------------------------------------

\draw[special]
  (t1)--(b1)
  (t2)--(b1)
  (t3)--(b2)
  (t4)--(b3)

  (t5)--(b1)
  (t6)--(b2)
  (t7)--(b3)
  (t8)--(b4)
  (t9)--(b5)
  (t10)--(b6);

%------------------------------------------------
% Main vertices and reversed labels
%------------------------------------------------

% Upper labels
\path
  (t1)  node[main vertex, label={[bit label]above:11100}] {}
  (t2)  node[main vertex, label={[bit label]above:11010}] {}
  (t3)  node[main vertex, label={[bit label]above:10110}] {}
  (t4)  node[main vertex, label={[bit label]above:01110}] {}

  (t5)  node[main vertex, label={[bit label]above:11001}] {}
  (t6)  node[main vertex, label={[bit label]above:10101}] {}
  (t7)  node[main vertex, label={[bit label]above:01101}] {}
  (t8)  node[main vertex, label={[bit label]above:10011}] {}
  (t9)  node[main vertex, label={[bit label]above:01011}] {}
  (t10) node[main vertex, label={[bit label]above:00111}] {};

% Lower labels
\path
  (b1)  node[main vertex, label={[bit label]below:11000}] {}
  (b2)  node[main vertex, label={[bit label]below:10100}] {}
  (b3)  node[main vertex, label={[bit label]below:01100}] {}
  (b4)  node[main vertex, label={[bit label]below:10010}] {}
  (b5)  node[main vertex, label={[bit label]below:01010}] {}
  (b6)  node[main vertex, label={[bit label]below:00110}] {}

  (b7)  node[main vertex, label={[bit label]below:10001}] {}
  (b8)  node[main vertex, label={[bit label]below:01001}] {}
  (b9)  node[main vertex, label={[bit label]below:00101}] {}
  (b10) node[main vertex, label={[bit label]below:00011}] {};

\end{tikzpicture}
\caption{$L_3(C_5,1,0)$ graph obtained by attaching the yellow $5$-cycles to special magenta edges of the middle layer of $Q_5$}
\end{figure}\label{fig}

\begin{Theorem}\label{grapes}
        Let $r$ be a natural number and $G, F_1,\dots,F_s$ be graphs such that $G \xrightarrow{r} \{F_1,\dots,F_s\}$. Let $H$ be a graph and $(a_i,b_i)$ be a pair of adjacent vertices of $F_i$, $i \in \{1, \ldots, s\}$. Suppose there is a sufficiently large $k$ such that a copy of $H$ is a subgraph of $L_k(F_i,a_i,b_i)$ for all $i \in \{1, \ldots, s\}$.
    Then there exists $n \in \N$ such that $G^{\square n} \xrightarrow{r} H$.
\end{Theorem}

Informally speaking, the above theorem claims that a graph obtained from any layered graph $L$ by attaching some graphs $F$'s (with nice Ramsey properties in $G$) to special edges of $L$ is Ramsey in the product  $G^{\square n}$. A result in \cite{ALS} claims that a graph obtained from any zero-Tur\'an density (in the hypercube) graph $Z$ by attaching some graphs $F$'s (with nice Ramsey properties in $G$) to all edges of $Z$ is Ramsey in the product  $G^{\square n}$. We compare these results and give more examples in Section \ref{conclusions}.

One of the motivating applications of Cartesian powers of graphs is in Euclidean graph theory and unit-distance graphs.
For references on Euclidean Ramsey theory, see \cite{ALS}.
A graph $G$ is a {\it unit-distance graph} in some given metric space $M$ if there is an embedding of the vertices of $G$ in $M$ such that two vertices of $G$ are adjacent if and only if their images are at distance one in $M$.
Horvat and Pisanski \cite{HP} showed that the Cartesian product of two unit distance graphs is a unit distance graph in the same metric space. Thus Cartesian powers of unit distance graphs are large unit distance graphs having nice Ramsey properties, which in turn imply Ramsey-type properties for colorings of the plane.  One of the main functions studied in the Euclidean graph theory is $\chi=\chi_F(\mathbb{R}^n)$, where $F$ is a unit distance graph and $\chi$ is the smallest number of colors used on the points from $\mathbb{R}^n$ such that no copy of $F$ is monochromatic. Here, a copy of $F$ in $\mathbb{R}^n$ is a set of $V(F)$ points corresponding to the vertices of $F$ such that any two points from this set corresponding to two adjacent vertices of $F$ are at distance $1$. Specifically, $\chi_{K_2}(\mathbb{R}^d)$ is the chromatic number of the space, also simply denoted by $\chi(\mathbb{R}^d)$.
Theorem \ref{arrow_layered} implies the following.

\begin{Corollary} \label{cor:Euclid}
    For any layered graph $H$ and any natural number $n\geq 2$, 
$\chi_H(\mathbb{R}^n)= \chi (\mathbb{R}^n)$.
\end{Corollary}

To see why the corollary holds, observe that if $\chi_{K_2} (\mathbb{R}^n)=r+1$ then $\chi(G)=r+1$ for some unit-distance graph $G$. Then by Theorem \ref{arrow_layered}, we have that $G^{\square n} \xrightarrow{r} H$. Since $G^{\square n}$ is a unit-distance graph, we have that $\chi_H(\mathbb{R}^n)> r$.  As a special case,  \Cref{cor:Euclid} implies that $\chi_{C_6}(\mathbb{R}^2)=\chi_{C_{10}}(\mathbb{R}^2)=\chi(\R^2)\geq 6$ \cite{OAI}. This answers Question 2 from~\cite{ALS}, which was later reiterated by Moore and Sagdeev~\cite{MS}.

 The paper is structured as follows. In Section~\ref{sec:def} we introduce the necessary preliminaries. Section~\ref{lem:layer} gives a key Ramsey-type statement on the discrete cube $[m]^n$ that we call Cube Layered Lemma. 
 In Section~\ref{sec:grapes} we prove Theorem~\ref{grapes}. 
 In Section~\ref{sec:proofs} we derive Theorems~\ref{arrow_layered} and \ref{arrow_cycles} from Theorem \ref{grapes}.
 In Section~\ref{sec:other} we give some key coloring construction and prove Proposition~\ref{prop:necessary} as well as an additional result in the special regime when $\chi(G)=2r$. Section~\ref{sec:2-col} addresses the case $G^{\square n} \xrightarrow{2} H$, in particular when $G$ is assumed to have high chromatic number.   We conclude the paper in Section~\ref{conclusions} with final remarks and open problems.

%%%%%%%%%%%%%%%%%%%%%%%%%%%%%%%%%
%%%%%%%%%%%%%%%%%%%%%%%%%%%%%%%%%
\section{Preliminaries}
\label{sec:def}
%%%%%%%%%%%%%%%%%%%%%%%%%%%%%%%%%
%%%%%%%%%%%%%%%%%%%%%%%%%%%%%%%%%
For $n \in \N$, let $[n]\coloneq\{1,\dots,n\}$. 
Given $m,n\in\N$, we call the set $[m]^n$ an {\it $m$-ary $n$-cube}. For $\x=(x_1,\dots,x_n)\in [m]^n$ we call each $x_i$ the {\it coordinate} of $\x$ in {\it position} $i$ or the {\it $i$-th coordinate}.  We use  $\|\x\|_1$ to denote $x_1+\cdots+x_n$.
We also identify the elements of the cube $[m]^n$ with words of length $n$ over the alphabet $[m]$.
A {\it $d$-subcube} of the cube $[m]^n$ is a set $\{ (x_1, \dots, x_n) \in [m]^n:   x_i=y_i, i\not\in I\}$, 
where $I\subseteq [n]$ is a set of some $d$ indices and $y_i$'s are some fixed elements of $[m]$. We shall call $I$ the {\it set of variable positions} and $[n]\sm I$ the {\it set of fixed positions} of the subcube.
In the special case $d=1$, we call  $1$-subcubes of $[m]^n$ its \textit{slices}.
Recall that a graph $G$ is {\it layered} if it is a subgraph of an edge-layer of some hypercube. 
We call a graph \textit{cubical} if it is a subgraph of some hypercube.
A \textit{$G$-slice} of $G^{\square n}$ is a subgraph induced by the Cartesian product of $V(G)$ and $n-1$ one-element subsets of $V(G)$ in any order. There are exactly $n|V(G)|^{n-1}$ $G$-slices in $G^{\square n}$, each of which is isomorphic to $G$.
Note that if $G$ is a graph on the vertex set $[m]$ then a subgraph of $G^{\square n}$ induced by any  $d$-subcube of $[m]^n$ is isomorphic to $G^{\square d}$. In particular, a subgraph of $G^{\square n}$ induced by any slice of $[m]^n$ is a $G$-slice of $G^{\square n}$.

We shall need another lemma that is of independent interest about graphs embedded in Cartesian powers.
Let $H$ be a subgraph of $G^{\square n}$.  
We define the  {\it $1^{\rm st}$ projection } $H(1)$ of $H$ in $G^{\square n}$ as the  graph on vertex set $\{u\in V(G):   (u, v_2, \ldots, v_n)\in V(H) \mbox{ for some } v_2, \ldots, v_n\in V(G)\}$ and all edges $uw$,  such that $(u, v_2, \ldots, v_n)(w, v_2, \ldots, v_n)\in E(H)$ for some $v_2, \ldots,   v_n\in V(G)$. The $i^{\rm th}$ {\it projection} of $H$ in $G^{\square n}$ is defined similarly by considering the $i^{\rm th}$ vector coordinates instead of the first.

\begin{Lemma}\label{lem:embedding-cubical}
    Let $H$ be a subgraph of $G^{\square n}$ such that each projection of $H$ is cubical, then $H$ is cubical.
\end{Lemma}

\begin{proof}
Observe that $H$ is a subgraph of $H(1)\square H(2) \square \cdots \square H(n)$. Each $H(i)$ is a subgraph of $Q_m$, for some sufficiently large $m$.
Thus $H$ is a subgraph of $Q_m\square \cdots \square Q_m = Q_{mn}$. So, $H$ is cubical.
\end{proof}

%%%%%%%%%%%%%%%%%%%%%%%%%%%%%%%%%%%%%%%%%%%%%%%%%%%%%%%%%%%
%%%%%%%%%%%%%%%%%%%%%%%%%%%%%%%%%%%%%%%%%%%%%%%%%%%%%%%%%%%
\section{Cube Layered Lemma}\label{lem:layer}
%%%%%%%%%%%%%%%%%%%%%%%%%%%%%%%%%%%%%%%%%%%%%%%%%%%%%%%%%%%
%%%%%%%%%%%%%%%%%%%%%%%%%%%%%%%%%%%%%%%%%%%%%%%%%%%%%%%%%%%
Next we shall establish some further definitions and notations needed to state our 
Cube Layered Lemma, a result of independent interest. We shall use the lemma to prove Theorem \ref{grapes}. 

Let $m,n$ be positive integers and $t\in[m]$. We say that two elements $\x, \y\in [m]^n$ are \emph{$t$-equivalent}, and write $\x\sim_t \y$, if the tuples obtained from $\x$ and $\y$ by ``deleting" all coordinates equal to $t$ coincide. For example $(2,1,3,1,1,3) \sim_1 (1, 2, 3, 1, 1, 3)$ and $(2,1,3,1,1,3) \not\sim_1 (3, 1, 3, 1, 1, 2)$.

For $\x=(x_1,\dots,x_n)\in [m]^n$, we let $I_t(\x)$ be the set of indices $i \in [n]$ such that $x_i\neq t$. For $I \subseteq [n]$, let $\x|_I$ be a tuple obtained from $\x$ by restricting it to index set $I$, that is, if $\x=(x_1, \ldots, x_n)$, then $\x|_I = (x_{i_1}, \ldots, x_{i_{|I|}})$, where  $i_1< \cdots <i_{|I|}$, $I=\{i_1, \ldots, i_{|I|}\}$.  For example, $I \coloneq I_1((2, 1, 3, 1, 1, 3))=\{1,3,6\}$ and $(2, 1, 3, 1, 1, 3)|_I = (2, 3, 3)$.
Using this notation, we see that  two elements $\x=(x_1, \dots, x_n)$ and $\y=(y_1, \dots, y_n)$ from $[m]^n$ are \emph{$t$-equivalent} if $\x|_{I_t(\x)} = \y|_{I_t(\y)}$.  Note that if two distinct elements $\x$ and $\y$ are $t$-equivalent then $|I_t(\x)| =|I_t(\y)| \neq 0$.
 
The proof of the Cube Layered Lemma is a relatively
straightforward iteration of Ramsey's theorem. See for example~\cite{AL, AW} for the case when $m=2$ and~\cite{ALS} for the case when $m=3$. Similar arguments were also used by Leader, Russell, and Walters as part of their proof of~\cite[Theorem~3.1]{LRW12} and by Ivan, Leader, and Walters~\cite[Theorem~3]{ILW}. Nevertheless, we give a full proof here since, to the best of our knowledge, it did not explicitly appear in the published literature in this generality.

\begin{Lemma}[Cube Layered Lemma]
\label{lem:equivalent}
For any $ d', m, r \in \N$ and $t\in[m]$, there exists $ n' = n'(d',m,r)\in\N$ such that the following holds. For every $r$-coloring of $[m]^{n'}$, there is a $d'$-element subset $S' \subseteq[n']$ such that every two $t$-equivalent elements $\x, \y$ of the  $d'$-subcube $\{\z \in [m]^{n'}: I_{t}(\z)\subseteq S'\}$ receive the same color.
\end{Lemma}

\noindent
{\bf Remark.} When $m=2$, the lemma states that  in an $r$-colored hypercube, there is a subhypercube  such that each of its vertex layers is monochromatic. This gives the lemma its name.

\begin{proof}[Proof of~\Cref{lem:equivalent}]
When $m=1$, the cube $[m]^{n'}$ consists of one element and the statement holds trivially. Let now $m\geq2$ and $d',r\in\N$ be fixed integers. We shall
prove the lemma for $t=1$ as the argument for any other $t$ is identical. Let $n'=n'(d',m,r)\in\N$ be sufficiently large.

For a tuple $\aaa \in ([m]\sm\{1\})^{j}$, $1\leq j<n'$, and $I\subseteq [n']$ of size $j$, we let $\x^I[\aaa]=\x$ denote the unique element of $[m]^{n'}$ such that $I_1(\x)=I$ and $\x|_I = \aaa$. Namely, $\x^I[\aaa]$ is obtained by filling the designated positions $I$ with the tuple $\aaa$ and filling the remaining positions with $1$. For example, if  $n'=5$, then $\x^{\{2,4\}}[(3,2)]=(1, 3, 1, 2, 1)$.

Consider an arbitrary $r$-coloring $c:[m]^{n'} \to [r]$.  To find the desired subset $S' \subseteq [n']$, we shall iteratively apply hypergraph Ramsey theorem to certain auxiliary hypergraphs $\cH_1, \dots, \cH_{d'}$ with a new auxiliary edge-coloring defined on each of them. Recall that the Ramsey theorem claims that for any $r, j$, and  $d$ there is a sufficiently large $n$ such that any $r$-coloring of a complete $j$-uniform hypergraph on $n$ vertices contains a complete monochromatic sub-hypergraph on $d$ vertices. 

{\bf Step 1.} Let $\cH_1$ be the $1$-uniform complete hypergraph on the vertex set $[n']$. We assign to each edge $\{i\}$ of $\cH_1$ a color given by the set $\{(a,c(\x^{\{i\}}[a])):a\in[m]\sm\{1\}\}$. For example if $n'=4, m=3, r=6$,  $c((1, 2, 1, 1))=5$ and  $c((1, 3, 1, 1))=6$, then the color of $\{2\}$ is $\{(2, 5), (3,6)\}$. This gives an auxiliary coloring of the edges of $\cH_1$ using at most $r^{m-1}$ colors. By the pigeonhole principle, there is a large set of vertices of $\cH_1$, say $S_1\subseteq [n']$, inducing a monochromatic $1$-uniform clique under the auxiliary coloring. Observe that every two $1$-equivalent elements $\x, \y$ of the  $|S_1|$-subcube $\{\z \in [m]^{n'}: I_{1}(\z)\subseteq S_1\}$ such that $|I_1(\x)|=|I_1(\y)|=1$ satisfy $c(\x)=c(\y)$. 

Let $j\in\N$ with $2\leq j\leq d'$. Assume steps $1, \dots, j-1$ have been performed and a subset $S_{j-1} \subseteq [n']$ has been found with the property that every two $1$-equivalent elements $\x, \y$ of the  $|S_{j-1}|$-subcube $\{\z \in [m]^{n'}: I_{1}(\z)\subseteq S_{j-1}\}$ such that $|I_1(\x)|=|I_1(\y)|\le j-1$ satisfy $c(\x)=c(\y)$. 
 
{\bf Step $\boldsymbol{j}$.} Let $\cH_j$ be the $j$-uniform complete hypergraph on the vertex set $S_{j-1}$. We assign to each edge $I\in\binom{S_{j-1}}{j}$ of $\cH_j$ a color given by the set $\{(\aaa,c(\x^{I}[\aaa])):\aaa\in([m]\sm\{1\})^j\}$. This auxiliary edge-coloring of $\cH_j$ uses at most $r^{(m-1)^j}$ colors. Since $S_{j-1}$ is large enough (as $n'=n'(d',m,r)$ was initially chosen to be sufficiently large), by the hypergraph Ramsey theorem, there is a large set of vertices of $\cH_j$, say $S_j \subseteq S_{j-1},$ inducing a monochromatic $j$-uniform clique under the auxiliary edge-coloring. 
Observe that every two $1$-equivalent elements $\x, \y$ of the  $|S_j|$-subcube $\{\z \in [m]^{n'}: I_{1}(\z)\subseteq S_j\}$ such that $|I_1(\x)|=|I_1(\y)|=j$ satisfy $c(\x)=c(\y)$. In addition, since $S_j \subseteq S_{j-1}$, we also have $c(\x)=c(\y)$ for every two $1$-equivalent elements $\x,\y$ of this subcube such that $|I_1(\x)|=|I_1(\y)|\le j-1$.

After $d'$ steps, we obtain a set $S_{d'}\subseteq [n']$ with the property that $|S_{d'}|\geq d'$ and every two $1$-equivalent elements $\x, \y$ of the  $|S_{d'}|$-subcube $\{\z \in [m]^{n'}: I_{1}(\z)\subseteq S_{d'}\}$ such that $|I_1(\x)|=|I_1(\y)|\le d'$ satisfy $c(\x)=c(\y)$. The desired set $S'$ is obtained by taking any $d'$-element subset of $S_{d'}$.
\end{proof}

%%%%%%%%%%%%%%%%%%%%%%%%%%%%%%%%%

\section{Proof of \Cref{grapes}}\label{sec:grapes}

First, we shall state and prove another lemma needed in the proof of Theorem \ref{grapes}. 

Recall that $L_k(F, a,b)$ is a graph obtained by
``gluing" or ``attaching" copies of a graph $F$ to the middle layer $L_k$ of $Q_{2k-1}$  by identifying the edges corresponding to $ab$ with the respective special edges of $E^*$ having no $1$'s following the star, as defined before Theorem \ref{grapes}.

Let a graph $G$ on the vertex set $[m]$ contain $F$ as a subgraph. 
For a set $T \subseteq [n]$ of size $2k-1$, $a, b \in [m]$, $a\neq b$, and $\mathbf{y}\in [m]^{n-2k+1}$,  let $V = V(T, a, b, \mathbf{y})$ be the set of vertices in $G^{\square n}$ whose coordinates match $\mathbf{y}$ on $[n] \setminus T$, and whose coordinates in $T$ consist of $k-1$ copies of $a$,  $k-1$ copies of $b$, and one value from $V(F)$, possibly $a$ or $b$.
We denote by $G^{\square n}(T, a, b, \mathbf{y})$ the subgraph of $G^{\square n}$ induced by $V$.
For example, when $m=3$, $n=5$, $k=2$, $T=\{1, 2, 4\}$,  and $V(F)=\{a,b\}$,  
$V(T, 1,2, (3,3)) = \{( 1, 1, 3, 2, 3), (1, 2, 3, 1, 3), (2, 1, 3, 1, 3),(2, 1, 3, 2, 3), (2, 2, 3, 1, 3), (1,2,3,2,3)\}$.

\begin{Lemma} \label{L_hanging}
   The graph  $G'=G^{\square n}(T, a, b, {\bf y})$ contains a copy of $L_k(F, a, b)$ as a subgraph. Assume, in addition, that $G'$ is vertex-colored, such that, restricted to position in $T$, any  two vertices from 
   $V(T, a,b, {\bf y})$   that are $b$-equivalent are of the same color, and all vertices from  $V(T, a,b, {\bf y})$ that, restricted to $T$ belong to $\{a\}^{k-1} \times \{b\}^{k-1} \times V(F)$, are red. Then $G'$ contains a red copy of $L_k(F, a, b)$.
\end{Lemma}

\begin{proof}
Let $V'\subseteq [m]^{2k-1}$ be a set of $(2k-1)$-tuples, such that each element of this set is obtained from some $n$-tuple from $V$ by restricting it to positions in $T$. By identifying $a$ and $b$ with $1$ and $0$ respectively, we observe that $V'$ contains all binary vectors of length $2k-1$ with $k$ ones or $k$ zeros. Consequently, $G^{\square n}(T, a, b, \mathbf{y})$ contains a subgraph isomorphic to the middle layer $L_k$.
For each position $i$ from $T$, say without loss of generality $i=1$,  $V'$ contains all tuples from $V(F)\times \{a, b\}^{2k-2}$, with $(k-1)$ $a$'s and $(k-1)$  $b$'s in the last $2k-2$ positions. This set of vertices gives a copy of $L_k$ with copies of $F$ attached to all edges with direction $i=1$. Since this holds for any direction $i$ from $2k-1$ possible ones, we see that $V$ induces a graph corresponding to $L_k$ with a copy of $F$ glued to each of the edges in $L_k$. 
By taking only the copies of $F$ glued to the special edges $E^* \subseteq E(L_k)$, we see that the graph induced by $V$ contains a copy $L$ of  $L_k(F, a, b)$ as a subgraph.

Every vertex of $L$ restricted to $T$ has a form $(x_1, \ldots, x_i, c, b,b, \ldots, b)$, for some $c\in V(F)$ and $i\in \{k-1, k, \ldots, 2k-2\}$, such that $x_1, \ldots, x_i \in \{a, b\}$ and  excluding the $(i+1)$st position, there are exactly $(k-1)$ $a$'s and $(k-1)$ $b$'s. Each of these vertices is $b$-equivalent to ${\bf x}(c):= (a,\ldots,a, b,\ldots, b,c)$, where $c$ is in the last position. By assumptions ${\bf x}(c)$ is red for any $c\in V(F)$.
Thus all vertices of $L$ are red.
\end{proof}

\begin{proof}[Proof of~\Cref{grapes}]

Let $G$ be a graph on a vertex set $[m]$. Let $r$ be a natural number and $F_1,\dots,F_s$ be  graphs such that $G \xrightarrow{r} \{F_1,\dots,F_s\}$. Let $H$ be a graph and $(a_i,b_i)$ be a pair of adjacent vertices of $F_i$, $i \in [s]$. Fix a sufficiently large $k$ such that $H$ is a subgraph of $L_k(F_i,a_i,b_i)$ for all $i \in [s]$.

Denote the number of copies of $F_i$ in $G$ by $N_i$ for $i \in [s]$, and let $N = N_1+\dots+N_s$. Label all these copies by $F'_j$, $j \in [N]$. If $F'_j$ is a copy of $F_i$, then let $v_j,w_j$ be the vertices of $F_j'$ playing the roles of $a_i$ and $b_i$, respectively.

Let $n \in \N$ be sufficiently large.
Specifically, $n$ is bounded in terms of recursively defined integers $n_1, \ldots, n_N$, where  $n_N\geq 1$, $n_{j-1}\geq n'((2k-1)n_j, m, r)$, and $n \geq n'((2k-1)n_1, m, r)$, and $n'$  is as in Lemma \ref{lem:equivalent}.
Consider an arbitrary $r$-coloring $c$ of $[m]^{n}$, the vertex set of $G^{\square n}$, such that no copy of $H$ is monochromatic.

The main idea of the proof is as follows. We shall find a sequence of nested $n_j$-subcubes $M_j$, $j \le N$, such that in each slice of $M_j$, the subgraph corresponding to $F'_q$, for each $q\in [j]$ is not monochromatic. Since these subcubes are nested, we shall see that in each slice of the last subcube all subgraphs corresponding to $F'_j$, $j\in [N]$ are not monochromatic. This will give a contradiction to the Ramsey property $G \xrightarrow{r} \{F_1,\dots,F_s\}$, which shows that such a coloring $c$ does not exist and will complete the proof of \Cref{grapes}.

Recall that for $\x=(x_1,\dots,x_n)\in [m]^n$,  $I_t(\x)=\{ i \in [n]: x_i\neq t\}$.

\vspace{4mm}

\textbf{Step 1.} Apply \Cref{lem:equivalent} with $w_1$ playing the role of $t$ and $d_1\coloneqq n_1(2k-1)$ playing the role of $d'$. Since  $n\ge n'(d_1,m,r)$, this yields a $d_1$-element subset $S_1' \subseteq[n]$ such that every two $w_1$-equivalent elements $\x, \y$ of the  $d_1$-subcube $M_1' \coloneqq \{\z \in [m]^n: I_{w_1}(\z)\subseteq S_1'\}$ receive the same color.

Split $S_1'$ into $n_1$ ``segments" of $2k-1$  consecutive elements.
Let $S_1$ be the set consisting of the largest element from each segment, i.e., $|S_1|=n_1$.
Let $M_1$ be the  $n_1$-subcube of $M_1'$ defined as follows. Let the set of its variable positions be $S_1$.
For each segment defined above, each $\z \in M_1$  is equal to $v_1$ the $(k-1)$ positions from the segment and it is equal to $w_1$ on the next $(k-1)$ positions from the segment.

For instance, if $n=10, n_1=3, k=2$ and $S_1'=[10]\sm \{2\}$, then $S_1=\{4,7,10\}$,
$M_1'$ consists of $m^9$ elements of the form $(*,w_1,*,\boldsymbol{*},*,*,\boldsymbol{*},*,*,\boldsymbol{*})$. The set $S_1'$, corresponding to the stars, is split into consecutive triples  (segments) and $S_1$ is the set built of the last star from each triple (segment), see bold stars.
Each such segment is getting an assignment of elements $v_1w_1\boldsymbol{*}$.
Then  $M_1$ consists of $m^3$ elements of the form $(v_1,w_1,w_1,\boldsymbol{*},v_1,w_1,\boldsymbol{*},v_1,w_1,\boldsymbol{*})$.

We claim that in each slice of $M_1$, the subgraph corresponding to $F'_1$ is not monochromatic. Indeed, assume the contrary, namely that for some $i' \in S_1$, the subgraph $F$ corresponding to $F_1'$ in the slice of $M_1$ with variable position $i'$ is monochromatic, say red. 
We can assume without loss of generality that $i'$ is the first element of $S_1$, i.e., it is the $(2k-1)$st element of $S_1'$. 
  Let $T$ be the set of first $2k-1$  elements of $S_1'$.
 Observe that the vertices of  red subgraph $F$ restricted to positions in $T$ form $\{v_1\}^{k-1} \times \{w_1\}^{k-1} \times V(F_1')$. Let $\y$ be one of these vertices, restricted to positions in $[n] \backslash T$. Now \Cref{L_hanging} applied with $v_1,w_1,F_1'$ playing the roles of $a,b,F$, respectively, yields a red copy of $L_k(F_1',v_1,w_1)$ and thus a red copy of $H$ in $G^{\square n}$. Here the condition of \Cref{L_hanging} is indeed met: since the elements of $T$ are consecutive in $S_1'$, for any two vertices $\x, \z \in V(T, v_1, w_1, {\bf y})$, if $\x|_{T} \sim_{w_1} \z|_{T}$, then $\x \sim_{w_1} \z$ and so $c(\x)=c(\z)$ by construction.  However, recall that no copy of $H$ is monochromatic by the choice of $c$. This contradiction shows that in each slice of $M_1$, the subgraph corresponding to $F'_1$ is not monochromatic, as claimed. This completes Step 1.

\vspace{4mm}

Let $j\in\N$ with $2\leq j\leq N$. Assume steps $1, \dots, j-1$ have been performed and an  $n_{j-1}$-subcube $M_{j-1}$ of $[m]^{n}$ with the set $S_{j-1}$ of variable positions, $|S_{j-1}|=n_{j-1}$, has been found, such that in each slice of $M_{j-1}$, the subgraph corresponding to $F'_{j-1}$ is not monochromatic.

\vspace{4mm}

\textbf{Step $\boldsymbol{j}$.} The following procedure is almost identical to the one from Step 1. Nonetheless, we explicitly describe it for clarity. Apply \Cref{lem:equivalent} with $w_j$ playing the role of $t$, $d_j\coloneqq n_j(2k-1)$ playing the role of $d'$, $M_{j-1}$ playing the role of $[m]^{n'}$, and $S_{j-1}$ playing the role of the ground set $[n']$. Since $n_{j-1}\geq n'(d_j, m, r)$,   applying \Cref{lem:equivalent} yields a $d_j$-element subset $S_j' \subseteq S_{j-1}$ such that every two elements $\x, \y$ of the  $d_j$-subcube $M_j' \coloneqq \{\z \in M_{j-1}: I_{w_j}(\z) \cap S_{j-1}\subseteq S_j'\}$ satisfying $\x|_{S_{j-1}} \sim_{w_j} \y|_{S_{j-1}}$ receive the same color.

Split $S_j'$ into $n_j$ ``segments" of $2k-1$  consecutive elements.
Let $S_j$ be the set consisting of the largest element from each segment, i.e., $|S_j|=n_j$. 
Let $M_j$ be the  $n_j$-subcube of $M_j'$ defined as follows. Let the set of its variable positions be $S_j$.
For each segment defined above, each $\z \in M_j$  is equal to $v_j$ the $(k-1)$ positions from the segment and it is equal to $w_j$ on the next $(k-1)$ positions from the segment.

We claim that in each slice of $M_j$, the $j$  subgraphs corresponding to $F_1', \ldots, F'_j$ are not monochromatic.  
We have the $F_1', \ldots, F_{j-1}'$ are not monochromatic from previous steps and the fact that $M_j$ is a subcube of $M_{j-1}$. We only need to verify the statement about $F_j'$.
Assume the contrary, namely that for some $i' \in S_j$, 
the subgraph corresponding to $F_j'$ in the slice of $M_j$ with variable position $i'$ is monochromatic, say red. 
We can also assume that $i'$ is the first element of $S_j$, i.e., it is the $(2k-1)$st 
element of $S'_j$. 
We can assume that $i'=1$, since the other cases are analogous. Observe that the vertices of this red subgraph restricted to positions in the set $T$ of the first $2k-1$ elements of $S_j'$ form $\{v_j\}^{k-1} \times \{w_j\}^{k-1} \times V(F_j')$. Let $\y$ be one of these vertices, restricted to positions in $[n]\backslash T$. Now \Cref{L_hanging} applied with $v_j,w_j,F_j'$ playing the roles of $a,b,F$, respectively, yields a red copy of $L_k(F_j',v_j,w_j)$ and thus a red copy of $H$ in $G^{\square n}$. Here the condition of \Cref{L_hanging} is indeed met: since the elements of $T$ are `consecutive' in $S_j'$, for any two vertices $\x, \z \in V(T, v_j, w_j, {\bf y})$, if $\x|_{T} \sim_{w_j} \z|_{T}$, then $\x|_{S_{j-1}} \sim_{w_j} \z|_{S_{j-1}}$ and so $c(\x)=c(\z)$ by construction. However, recall that no copy of $H$ is monochromatic by the choice of $c$. This contradiction shows that in each slice of $M_j$, the subgraph corresponding to $F'_j$ is not monochromatic, as claimed. This completes Step $j$.

\vspace{4mm}

After $N$ steps, we find a sequence of nested  $n_j$-subcubes $M_j$, $j \le N$, such that in each slice of $M_j$, the subgraph corresponding to $F'_j$ is not monochromatic. Since these subcubes are nested, we see that any subgraph corresponding to some $F'_j$ in each slice of the last subcube is not monochromatic. However, every $G$-slice of $G^{\square n}$ contains a monochromatic subgraph corresponding to some $F'_j$ by the assumption that $G \xrightarrow{r} \{F_1,\dots,F_s\}$. 
This contradiction shows that, provided $n$ is sufficiently large, every $r$-coloring of $[m]^n$ contains a monochromatic copy of $H$. This completes the proof of \Cref{grapes}.
\end{proof}

%%%%%%%%%%%%%%%%%%%%%%%%%%%%%%%%%
%%%%%%%%%%%%%%%%%%%%%%%%%%%%%%%%%
\section{Proofs of Theorems \ref{arrow_layered}  and \ref{arrow_cycles} } \label{sec:proofs}
%%%%%%%%%%%%%%%%%%%%%%%%%%%%%%%%%
%%%%%%%%%%%%%%%%%%%%%%%%%%%%%%%%%

\begin{proof}[Proof of Theorem \ref{arrow_layered}]

We need to show that for any $H$ that is a layered graph, any natural number  $r$,  and any  $G$ graph of chromatic number $r+1$, $G^{\square n} \xrightarrow{r} H$, for sufficiently large $n$.
We see that $G\xrightarrow{r} K_2$ and $H$ is a subgraph of the middle layer graph $L_k$ for some $k$.
Thus $H$ is a subgraph of $L_k=L_k(F, a, b)$, for $F=K_2$ on vertices $a, b$. Thus the result follows from Theorem \ref{grapes}.
\end{proof}

\begin{proof}[Proof of Theorem \ref{arrow_cycles}]
Let  $G$ be a graph on $m$ vertices and $r,\ell$ be integers.

For the first part of the theorem, assume that $G \xrightarrow{r}\{C_3, C_5, \dots, C_{2\ell+1}\}$ and let  $s\ge\ell+2$. Let $k$ be sufficiently large compared to $s$.
We see that any edge of the middle layer $L_k$ is contained in an even cycle of length $6, 8, \ldots, 2s-2$, and $2s$.
Consider a union $Q$ of  $L_k$ and $C_{2i+1}$, $1\leq i\leq \ell$, that share exactly one edge.
It contains a union of $C_{2s- 2i +2}$ and  $C_{2i+1}$, that share exactly one edge. This union in turn contains a $C_{2s+1}$ as desired. 
Since $Q$ is a subgraph of $L_k(C_{2i+1}, a, b)$, where $ab$ is any edge of $C_{2i+1}$, Theorem \ref{grapes} implies the result.

For the second part of the theorem, assume that $\chi(G)>2r$. Then we have that coloring vertices of $G$ in $r$ colors results in a monochromatic subgraph with chromatic number greater than $2$, that then contains an odd cycle. This cycle has length at most $m=|V(G)|$. Thus $G \xrightarrow{r}\{C_3, C_5, ,\dots, C_{2\ell+1}\}$, for $\ell= \lfloor (m-1)/2 \rfloor$. Let $2s +1 \geq m+4$, then  $s\geq \lfloor (m-1)/2 \rfloor +2$, i.e., $s\geq \ell+2$. By the first part of the theorem, if $s\geq \ell +2 $, then $G^{\square n} \xrightarrow{r} C_{2s+1}$. 
\end{proof}

%%%%%%%%%%%%%%%%%%%%%%%%%%%%%%%%%
%%%%%%%%%%%%%%%%%%%%%%%%%%%%%%%%%
\section{Necessary conditions for unavoidable graphs}
\label{sec:other}
%%%%%%%%%%%%%%%%%%%%%%%%%%%%%%%%%
%%%%%%%%%%%%%%%%%%%%%%%%%%%%%%%%%

In this section, we investigate the opposite direction of our main results. More specifically, for a graph $G$ and a positive integer $r$, we are looking for necessary conditions a graph $H$ must satisfy if $G^{\square n} \xrightarrow{r} H$ for some $n \in \N$. To this end, for each $n \in \N$ we construct an $r$-coloring of the vertices of $G^{\square n}$ in which all monochromatic components do not satisfy these conditions. All our explicit colorings are special cases of the following general construction.

\begin{Construction}\label{sum_mod_r}
    Let $G$ be a graph and $n, r, t$ be positive integers and  $c':V(G) \to [r]$ be a coloring. For $\x\in V(G^{\square n})$ let  $c'(\x)=(c'(x_1),\dots,c'(x_n))$. Define an $(t, r, c')$-coloring of the vertices of $G^{\square n}$ by 
\begin{equation*}
   c(\x)= \begin{cases}
   0, & \mbox{if } 0\leq \|c'(\x)\|_1 <t,\\
   1, & \mbox{if }  t\leq \|c'(\x)\|_1 <2t,\\
   &\vdots\\
   r-1, & \mbox{if } (r-1)t\leq \|c'(\x)\|_1 <rt,
   \end{cases}
\end{equation*}
where the addition is modulo $rt$.

\end{Construction}

\begin{Lemma} \label{Obs:1}
Any monochromatic component of $G^{\square n}$ under a $(1,r,c')$-coloring $c$ is a Cartesian product of some monochromatic components of $G$ under $c'$. In particular, if each monochromatic component of $G$ under $c'$ is $k$-partite (or cubical), then each monochromatic component of $G^{\square n}$ under $c$ is $k$-partite (or cubical, respectively).
\end{Lemma}

\begin{proof}
Let  $c':V(G) \to \mathbb [r]$ be a coloring. Any two adjacent vertices $\x = (x_1,\dots, x_n)$ and $\x' = (x'_1,\dots, x'_n)$ of $G^{\square n}$ differ in exactly one coordinate, say $x_j\neq x'_j$.  If $c(\x)=c(\x')$ for an $(1,r,c')$-coloring $c$, then $c'(x_j)=c'(x_j')$.
Let $V_i = \{v \in V(G): c'(v)=i\}$ and $G_i=G[V_i]$, $i \in [r]$. Each subgraph of $G^{\square n}$ induced by some color class under $c$
has the form $G_{i_1}\square\cdots \square G_{i_n}$ for some $i_1,\dots,i_n \in [r]$.
We conclude the proof by noting that the Cartesian product of $k$-partite graphs is $k$-partite and the Cartesian product of cubical graphs is cubical.
\end{proof}

We can now immediately obtain a necessary condition that $H$ is bipartite for a range of values of $\chi(G)$ and $r$.

\begin{Corollary} If $\chi(G)\leq 2r$ and $G^{\square n} \xrightarrow{r} H$, then $H$ is bipartite. If $\chi(G)> 2r$, then there is a non-bipartite graph $H$ such that $G^{\square n} \xrightarrow{r} H$.
\end{Corollary}

\begin{proof}
If $\chi(G)\leq 2r$, then there is an $r$-coloring $c':V(G)\to [r]$ with bipartite monochromatic components. Indeed, take a proper $\chi(G)$-coloring of $V(G)$ and split the color classes into $r$ groups of size at most $2$. Now \Cref{Obs:1} applied to this $c'$ shows that for each $n\in \N$, there is an $r$-coloring of the vertices of $G^{\square n}$ with bipartite monochromatic components. 
If $\chi(G)> 2r$,   \Cref{arrow_cycles} gives a non-bipartite graph $H$, namely a long odd cycle, and $n \in \N$ such that $G^{\square n} \xrightarrow{r} H$. 
\end{proof}

Thus $\chi(G)=2r$ gives a natural threshold for a necessary condition on $H$ to be bipartite. Now, instead of the condition on $H$ being bipartite, we consider layered $H$ and prove the respective Proposition~\ref{prop:necessary}.
Recall that it claims for positive integers $s$ and $r$, $s/2<r<s$, that
 $$\mbox {if } C_{2s+1} ^{\square n} \xrightarrow{2} H  \mbox{ or } K_s^{\square n} \xrightarrow{r} H, \mbox{ then }H \mbox{ is layered}.$$

\begin{proof}[Proof of Proposition \ref{prop:necessary}]
(1) Let $V(C_{2s+1})=[2s+1]$. Let  $c':V(C_{2s+1}) \rightarrow [3]$ be a unique proper vertex-coloring such that $c'(1)=1, \ c'(2)=2,\ c'(3) = 3$ and $c'(x)\neq 2$ for any $x\in [2s+1]\sm\{2\}$. Consider a $(2,2,c')$-coloring  $c$ of the vertices of $C_{2s+1}^{\square n}$. In other words, for $i\in \{0,1\}$ and $\x=(x_1,\dots,x_n) \in [2s+1]^n$, we have $c(\x)=i$ if and only if $\|c'(\x)\|_1 \in \{2i, 2i+1\} \pmod {4}$.

Let $\x, \y \in [2s+1]^{n}$ be adjacent and both have color $i$ under $c$. 
In particular, $\x$ and $\y$ differ in exactly one coordinate, say  $j^{\mathrm{th}}$, and thus $\|c'(\x)-c'(\y)\|_1 = |c'(x_j)-c'(y_j)| \in \{1,2\}$. 
Without loss of generality, $\|c'(\x)\|_1 \equiv 2i+1 \pmod {4}$ and $\|c'(\y)\|_1 \equiv 2i \pmod {4}$.
Hence we have $\|c'(\x)-c'(\y)\|_1= 1$, and thus either $\{x_j,y_j\}=\{2,1\}$ or $\{x_j,y_j\}=\{3,2\}$. 

Now let $C$ be a monochromatic component of $C_{2s+1}^{\square n}$.
Since every monochromatic edge joins a vertex with $\|c'(\cdot)\|_1\equiv 2i \pmod 4$ to one with $\|c'(\cdot)\|_1\equiv 2i+1 \pmod 4$, the value of $\|c'(\cdot)\|_1 \pmod 4$ attains only the two consecutive values $4k+2i$ and $4k+2i+1$, for some integer $k$,  on the vertices of $C$.
Moreover, the argument above shows that, in each of the $n$ projections of $C$, the only possible edges are $\{1,2\}$ and $\{2,3\}$.

Let $f$ be any injective mapping from $[2s+1]$ to the vertex set $\{0,1\}^d$ of the binary hypercube $Q_d$ for some $d$ such that  it preserves the weight, i.e. that $\|f(x)\|_1=c'(x)$ for each $x \in [2s+1]$, and $f(2)$ is adjacent to both $f(1)$ and $f(3)$. Note that the function $f$ need not preserve all edges of $C_{2s+1}$, only these two possible projection edges.
To see that this function exists, suppose that $d$ is sufficiently large, map $2$ to any vertex from the second vertex layer, map $1$ and $3$ to any of the neighbors of $f(2)$ in the first and third layers, respectively, and injectively map each of the remaining $x \in [2s+1]$ arbitrarily to either the first or the third vertex layer depending on $c'(x)$. Now we define the map $g: V(C)\rightarrow V(Q_{nd})=\{0,1\}^{nd}$ by letting $g(\x)$ 
to be a concatenation of $f(x_1), f(x_2), \ldots, f(x_n)$ into a single binary vector. For each vertex $\x$ in $C$, we have $\|f(\x)\|_1=\|c'(\x)\|_1$. Since 
$\|c'(\x)\|_1$ takes two consecutive values, $g$ maps vertices of $C$ into two consecutive vertex layers of $Q_{nd}$. Since the map $g$ preserved adjacencies and is injective, it maps $C$ into an edge-layer of $Q_{dn}$. Thus $C$ is layered, as desired.

(2) Consider a $(2,r,c')$-coloring $c$ of the vertices of $K_s^{\square n}$, where $c'$ is a proper $s$-coloring of $K_s$. In other words, if $[s]$ is the vertex set of $K_s$, then for each $\x \in [s]^n$ and $0\le i < r$, we have $c(\x)=i$ if and only if $\|\x\|_1 \in \{2i, 2i+1\} \pmod {2r}$. It remains to show that monochromatic components of $K_s^{\square n}$ under $c$ are layered.

Let $\x, \y \in [s]^{n}$ be adjacent and both have color $i$ under $c$.  Then $\|\x\|_1=2i + \epsilon_x+ k_x\cdot 2r$ and $\|\y\|_1 = 2i + \epsilon_y + k_y\cdot 2r$, for some $k_x, k_y\in \mathbb Z$ and $\epsilon_x, \epsilon_y\in \{0,1\}$.
Since $\x$ and $\y$ differ in exactly one coordinate, $1 \le |\|\x\|_1-\|\y\|_1| \le s-1 < 2r-1$. We also have that $ |\|\x\|_1-\|\y\|_1| = |\epsilon_x- \epsilon_y + (k_x-k_y) 2r| $.
This implies that $k_x=k_y$ and $\epsilon_x=1-\epsilon_y$.
 Therefore, the function $\|\cdot\|_1$ attains only 2 consecutive values on each of the monochromatic components of $K_s^{\square n}$.

Let $C$ be a monochromatic component of $K_s^{\square n}$, $i \in \N$, and $k\in \mathbb{Z}$ be such that $\|\x\|_1\in \{2i + 2kr, 2i+1 + 2kr\}$ for each $\x \in V(C)$. To show that $C$ is layered, we map each $j \in [s]$ to the binary vector of length $s$ with $s-j$ zeros followed by $j$ ones. This naturally maps the vertices of $K_s^{\square n}$ to the vertices of $Q_{sn}$. This mapping preserves the function $\|\cdot\|$, e.g., the vertices of $C$ are mapped into the $(2i + 2kr)^{\mathrm{th}}$ and $(2i+1 +2kr)^{\mathrm{th}}$ vertex layers of $Q_{sn}$. Moreover, this mapping preserves the adjacencies  since consecutive elements of $[s]$ are mapped into binary vectors at Hamming distance one. Thus $C$ is a layered graph, as desired.
\end{proof}

Note that the condition $s/2 <r<s$ on $r$  in the statement of \Cref{prop:necessary} is almost best possible for $G=K_s$ in the following sense.

\begin{Corollary} If $r \geq s$ and  $K_s^{\square n} \xrightarrow{r} H$, then $H$ is edgeless. If $r< s/2$, there is a non-layered $H$, such that $K_s^{\square n} \xrightarrow{r} H$.
\end{Corollary}

\begin{proof}
If $r\ge s$, then for any $n\in \N$, there is a proper $r$-coloring of the vertex set of $K_s^{\square n}$ because $\chi(K_s^{\square n})=\chi(K_s)=s$. If $r<s/2$, then for a non-layered graph $H=C_3$, the pigeonhole principle shows that $K_s^{\square n} \xrightarrow{r}C_3$ already for $n=1$.
\end{proof}

However, the case $r=s/2$ remains elusive. ChatGPT-6 Astra found a non-layered graph $H$ such that $K_{4}^{\square n} \xrightarrow{2}H$ for some $n \in \N$. The proof is ad hoc and does not seem to generalize to larger values of $r$. So we do not incorporate it into the present paper, and only outline it briefly in \Cref{conclusions}. We suspect that, for each $r\ge 2$, there exists a non-layered graph $H$ such that $K_{2r}^{\square n} \xrightarrow{r}H$ for some $n \in \N$. Here we only show that such a graph, if exists, must satisfy the following restrictive properties. Let $\Theta_{3,3,3}$ be a graph formed by three paths between two distinct vertices that share only their endpoints. Note that the graph $\Theta_{3,3,3}$ is the smallest cubical, non-layered, $C_4$-free graph.  

\begin{Proposition} \label{2r=s}
    Let $H$ be a graph and $r\ge2$ be an integer. If there exists $n$ such that $ K_{2r}^{\square n}\xrightarrow{r}H$, then $H$ is cubical, $C_4$-free, and $\Theta_{3,3,3}$-free.
\end{Proposition}

\begin{proof}
To prove the statement, we show an $r$-coloring of the vertices of $K_{2r}^{\square n}$ in which every monochromatic component is cubical, $C_4$-free, and $\Theta_{3,3,3}$-free. We employ the construction described in Construction \ref{sum_mod_r}. 
Specifically, let $c$ be a $(2, r, c')$-coloring of $V(K_{2r}^{\square n})$, where  $c'$ is a proper vertex coloring of $K_{2r}$.
So if $[2r]$ is the vertex set of $K_{2r}$, then for each $\x \in [2r]^n$ and $0\le i < r$, we have $c(\x)=i$ if and only if $\|\x\|_1 \in \{2i, 2i+1\} \pmod {2r}$

Consider two adjacent vertices $\x, \x'$  of the same color in $K_{2r}^{\square n}$  under $c$. By the definition of the Cartesian product, $\x$ and $\x'$ differ in exactly one coordinate $j \in [n]$. Let $x_j$ and $x'_j$ be the values of this coordinate. We have that $x_j-x'_j \equiv \pm 1 \pmod{2r}  $. Let $H$ be a monochromatic component of $K_{2r}^{\square n}$. We see that each projection of $H$ in $K_{2r}^{\square n}$ is contained in a cycle $C_{2r}$ formed by vertices $(1,\dots, 2r)$.
Thus, each projection is a cubical graph. By Lemma \ref{lem:embedding-cubical}, $H$ is cubical as well.

Next we prove that each monochromatic component is $C_4$-free.
Suppose $\x_1, \x_2, \x_3,\x_4$ form a monochromatic $C_4$ in this particular order that does not lie entirely in any $K_{2r}$-slice. 
We know that $\|\x_j\|_1\in \{2i,2i+1\} \pmod{2r}$ for some $i$, and $j \in [4]$.
Without loss of generality, assume that $\x_1$ has the smallest sum of coordinates $\|\x_1\|_1$ among $\x_1, \x_2, x_3,$ and $\x_4$. Then we get $\x_2$ by increasing the $k_1$th coordinate of $\x_1$ by $a$ for some $a < 2r$, namely, $\|\x_2\|_1= \|\x_1\|_1+a$. Note that $a=1$ if $\|\x_1\|_1 \equiv 2i \pmod{2r}$ and $a=2r-1$ if $\|\x_1\|_1 \equiv 2i+1 \pmod{2r}$, by minimality of $\|\x_1\|_1$.
Similarly, we get $\x_4$ by increasing the $k_2$-th coordinate of $\x_1$ by the same value $a$, so $\|\x_4\|_1= \|\x_1\|_1+a$.
The fourth vertex must complete the cycle, so we get $\x_3$ by increasing both the $k_1$-th and the $k_2$-th coordinates of $\x_1$ by $a$, so
$\|\x_3\|_1 = \|\x_1\|_1+2a$. However, $\|\x_1\|_1,\|\x_1\|_1+a, \|\x_1\|_1+2a \in \{2i, 2i+1\} \pmod{2r}$ is a contradiction, as $a \in \{1,2r-1\}$. 
If a copy of $C_4$ lies in a $K_{2r}$-slice of a monochromatic component $H$, then this copy is contained in some projection of $H$ in $K_{2r}^{\square n}$. However, by the above argument we know that every projection of $H$ is contained in a cycle $C_{2r}$. So we have $r=2$, and this $C_4$ uses all vertices of the slice. By construction, no slice is monochromatic, a contradiction.

Finally suppose there is a copy of $\Theta_{3,3,3}$ in color $j$. Let vertices $\x,\y$ in this copy be connected by three internally vertex-disjoint paths of length $3$. Every vertex $\textbf{v}$ satisfies $\|\textbf{v}\|_1 \in \{2j,2j+1\} \pmod{2r}$, and adjacent vertices belong to different congruence classes. Without loss of generality, let $\|\x\|_1 \equiv 2j$ and $\|\y\|_1 \equiv 2j+1 \pmod {2r}$.
Since $\x$ and $\y$ are joined by a path of length~3, their Hamming distance is at most $3$. 
 The distance cannot be $1$, as otherwise $\x$ and $\y$ would be adjacent, creating a monochromatic $C_4$. 
Suppose that the Hamming distance between $\x$ and $\y$ is $2$, and let $k_1$ and $k_2$ be the coordinates in which they differ. Any $\x$--$\y$ path of length~$3$ must change only these two coordinates, with one of them changed twice. The repeated coordinate cannot be changed in two consecutive steps, since the second change would undo the first and revisit a previously visited vertex. Hence the coordinate changes must alternate, so the only possible orders are $k_1k_2k_1$ and $k_2k_1k_2$. Therefore there are at most two $\x$--$\y$ paths of length~$3$, a contradiction.
Hence $\x$ and $\y$ differ in exactly three coordinates.
Fix one $\x$--$\y$ path, and let $k_1,k_2,k_3$ be the coordinates changed by its first, second, and third edges. The middle edge goes from congruence class $2j+1$ to $2j$, so the $k_2$-th coordinate changes by $-1 \pmod{2r}$ .
Any length-$3$ path from $\x$ to $\y$ must change these same three coordinates, one at a time, in some order. However, changing the $k_2$-th coordinate first or last produces a vertex $v$ with $\|\textbf{v}\|_1 \notin \{2j,2j+1\} \pmod{2r}$, so $k_2$ must always be changed second. Thus the only possible orders are $k_1k_2k_3$ and $k_3k_2k_1$, yielding at most two internally vertex-disjoint paths of length~$3$, contradicting the existence of a copy of $\Theta_{3,3,3}$.
\end{proof}

%%%%%%%%%%%%%%%%%%%%%%%%%%%%%%%%%
\section{Forcing monochromatic $H$ in $2$-colorings of Cartesian powers} \label{sec:2-col}
%%%%%%%%%%%%%%%%%%%%%%%%%%%%%%%%%

The smallest non-trivial number of colors used on the vertices of a Cartesian power is two. Even in this case it is not clear what condition on $G$ gives a positive Ramsey-type result.

\begin{Question}\label{q1}  Let $H$ be a graph. Is it true that there is a constant $\chi$ such that  for any $G$ of chromatic number $\chi$ and sufficiently large $n$ we have that 
   $G^{\square n} \xrightarrow{2} H?$ 
   \label{ques:chi-c4}
\end{Question}

We observe that we can restrict our attention to cubical non-layered graphs $H$.  
For that,  consider a graph $G$ of girth larger than $|H|$ and arbitrarily large chromatic number.  Since any projection of $H$ in $G^{\square n}$ has order at most $|H|$, it is a forest, that is a cubical graph. Then by Lemma \ref{lem:embedding-cubical} $H$ is cubical.  If $H$ is layered, then the answer to the above question is yes for $\chi \geq 3$ by Theorem \ref{arrow_layered}.

For the rest of this section we shall treat $H=C_4$, the smallest cubical non-layered graph. 

Observe that $K_7^{\square 1} \xrightarrow{2} C_4$ easily follows by the pigeonhole principle. Similarly one can check that $K_5^{\square 2} \xrightarrow{2} C_4$. Indeed, in any red/blue coloring of $[5]^2$, there are three slices, say w.l.o.g. $\{i\}\times [5]$, $i\in \{1, 2, 3\}$ that each have three red vertices, thus four of these form a red $C_4$.  In the proposition below, we shall show that $\chi=5$ is not sufficient in Question \ref{q1}. On the other hand, if the chromatic number is large as a function of the odd girth or order,  we shall show that $C_4$ is arrowed by  $G^{\square n}$ in two colors.  While Question \ref{q1} remains open in general even for $H=C_4$, we provide some partial results.

Let the length of a shortest odd cycle in a graph $G$ be denoted $g_o(G)$ and called its {\it odd girth}.
For a group $\Gamma=(X, \circ)$ and set $S\subseteq X$, 
the {\it Cayley graph} $Cay(\Gamma, S)$ has vertex set $X$ and edge set $\{\{x, x\circ s\}: s\in S, x\in X\}$. 

\begin{Proposition} Let $G$ be a graph.
\vspace{-0.3cm}
   \begin{itemize}
       \item 
 If $ \chi(G)>2g_o(G)$, then $G^{\square 2} \xrightarrow{2} C_4$.
    In particular, if $|G|< (\chi(G)^2 - 3 \chi(G) +6)/4$, $G^{\square 2} \xrightarrow{2} C_4$.
    \item 
      If $S = \{1,2,11,16\}$ and $G = Cay(\mathbb{Z}_{58}, S)$,  then $\chi(G)=5$ and $G^{\square n} \not\xrightarrow{2} C_4$.
    \end{itemize}
\end{Proposition}

\begin{proof}
Assume first that $ \chi(G)>2g_o(G)$. Let $c$ be an arbitrary red-blue vertex coloring of $G^{\square 2}$. We denote by $G\times v$ and $v\times G$ the respective slices of $G^{\square 2}$, that are subgraphs induced by $V(G)\times \{v\}$ and $\{v\}\times V(G)$, respectively.  Let $C \subseteq G$ be a shortest odd cycle in $G$ on $g_o(G)$ vertices.  Since $C$ is an odd cycle, for each $v\in V(G)$,  $C\times v$ has an edge whose endpoints have the same color under $c$. Define a new coloring $c'$ of $G$ as follows: for each $v \in V(G)$ let $c’(v):= (e, s)$ where $e \in E(C)$ and $s \in  \{\mathrm{red}, \mathrm{blue}\}$ such that in the subgraph  $G \square\{v\}$  of $G^{\square 2}$, the edge corresponding to $e$ has endpoints of color $s$ in $c$. This defines a coloring of $G$ in at most $2|C|$ colors. Since $2|C| <  \chi(G)$ ,  the coloring $c'$ is not proper, and hence there exists a monochromatic edge $e’$  in $c’$. This yields a monochromatic $C_4$ in $G^{\square 2}$ consisting of two edges corresponding to $e$ and two edges corresponding to $e'$.
    
    Next we shall show that
    $g_o \leq (2|G|-2)/(\chi(G)-1)+1$.
    Let $H \subseteq G$ be a $\chi(G)$-vertex-critical subgraph of $G$ and fix a vertex $v \in V(H)$. Consider the graph $H - v$. By criticality, we have $\chi(H-v) = \chi(G)-1$. 
    Take a proper vertex coloring of $H-v$ with $\chi(H-v) = \chi(G)-1$ colors. The two smallest color classes contain at most $(2|H-v|)/(\chi(G) -1)$ vertices. Then the subgraph of $H$ induced by these vertices together with $v$ is not bipartite, and thus contains an odd cycle $C'$. Hence we have
    $g_o \leq |C'| \leq (2|H-v|)/(\chi(G)-1) +1 \leq (2|G|-2)/(\chi(G)-1) +1.$
Now, we have that $2g_0\leq (4|G|-4)/(\chi(G)-1) +2 $ and we assumed that 
    $|G|< (\chi(G)^2 - 3 \chi(G) +6)/4$, we have that  $ 2g_o <\chi(G)$, so we are done by the first statement. This proves the first part of the proposition.
    \\

Now we consider the Cayley graph $G$ for the second item of the proposition.
    Let $c: V(G^{\square n}) \rightarrow \{0,1\}$ defined as follows.
     Let $c(x_1,...,x_n)=F(x_1+...+x_n)$, where    
    $F: \mathbb{Z}_{58}\rightarrow \{0,1\}$ is a function such that $F(x)=1- F(29+x)$ and is given by its values on $0, 1, \ldots, 28$ in order:
    $0 1 1 0 0 1 0 1 1 0 1 0 0 1 0 1 1 0 0 1 1 0 1 1 0 1 1 0 1$.
    A computer verification gives that $F$ has no monochromatic copy of $C_4$ in $G$ and thus $c$ has no monochromatic copy in $G^{\square n}$ for $n=1$. Moreover, we can observe that a shifted coloring $F_a(x)=F(a+x)$ also has no monochromatic copy of $C_4$ for any $a$.
    
   Consider an arbitrary copy $C$ of $C_4$ in $G^{\square n}$, $n\geq 2$.
   Then $n-1$ or $n-2$ coordinates of the vertices of $C$ are fixed, and the remaining coordinates are ``active". If there is only one active coordinate, say the $n$th one,  we see that $C$ is in a $G$-slice of $G^{\square n}$ obtained by fixing $x_1=x_1^0, x_2=x_2^0, \ldots, x_{n-1}^0$. Then for $a=x_1^0+\cdots +x_{n-1}^0$, we have 
   for any vertex $x=(x_1^0, \ldots, x_{n-1}^0, x_n)$ of $C$ that $c(x)= F(a+x_n)$. Thus $C$ is not monochromatic by a property of $F$.
If there are two active coordinates, say the $n$th and $(n-1)$st, 
let them be  $(a_1,b_1), (a_2,b_1), (a_2,b_2), (a_1,b_2)$ where $a_1-a_2 = s_1 \in S$ and $b_1-b_2 = s_2 \in S$, and let the sum of all other, fixed coordinates be $a$. 
Then the colors on the vertices of $C$ are 
$F(a+a_1+b_1), F(a+a_2+b_1), F(a + a_2 +b_2)$, and $F(a+a_1+b_2)$.
We see that the four elements $a+a_1+b_1, a+a_2+b_1, a + a_2 +b_2$, and $a+a_1+b_2$ form a $C_4$ in the Cayley graph $G$, that is not monochromatic under $F$. Thus $C$ is not monochromatic under $c$. A computer verification also shows that  $\chi(G)=5$.\end{proof}

We note here that our program exhaustively searched for such Cayley graphs with chromatic number 5 in all Abelian groups of order less than 58 and showed that none exist. We wonder if extending the search further can lead to similar examples with a larger chromatic number. If true, this would also confirm in a strong form Conjecture~7 of Mubayi and R\"odl from~\cite{MubRod}.

%%%%%%%%%%%%%%%%%%%%%%%%%%%%%%%%%%%%%%%
%%%%%%%%%%%%%%%%%%%%%%%%%%%%%%%%%%%%%%%
\section{Concluding remarks and open questions}\label{conclusions}

In this paper, we investigate the conditions on graphs $G$ and $H$ and integer $r$,  that are necessary or sufficient for the following Ramsey statement:
$G^{\square n} \xrightarrow{r} H$, for some $n\in \mathbb{N}$. \Cref{arrow_layered} states that for any graph $G$ of chromatic number $r+1$ and layered graph $H$, we have $G^{\square n} \xrightarrow{r} H$, some $n\in \mathbb{N}$. \Cref{prop:necessary} states that for some host graphs $G$, including cliques and odd cycles, this sufficient condition is also necessary. It would be interesting to describe all such graphs.

  \begin{Question} \label{Q_r+1}
  For which graphs $G$ of chromatic number $r+1$ does the fact that $ G^{\square n} \xrightarrow{r} H$ for some $n \in \N$ imply that $H$ is layered?
 \end{Question}

It is natural to ask about the dependence of $n$ in the statement of \Cref{arrow_layered} on $G,H$, and $r$. More specifically, our arguments give a tower-type upper bound on the minimum $n=n(s)$ such that $C_{2s+1}^{\square n} \xrightarrow{2} C_6$. A construction found by ChatGPT-6 Astra shows that $n(s)> \log_2s$ since for smaller dimensions $n$, there is even a 2-coloring of the vertices of $C_{2s+1}^{\square n}$ such that monochromatic components are trees. It would be interesting to reduce the gap between the bounds.

Here is the construction sketch. For $s\geq 2^{n}-1$, identify the vertices of $C_{2s+1}^{\square n}$ with $\Z_{2s+1}^n$. For $x=(x_1, \ldots, x_n) \in \Z_{2s+1}^n$, let $R(x)\in \{0, \ldots, 4s+1\}$ be such that $R(x)= \sum_{i=1}^n(2s+2^i)x_i \pmod {4s+2}$ and $c(x)= \lfloor R(x)/(2s+1) \rfloor$. One can verify that monochromatic components of this coloring are trees.

%\begin{Question}\label{q-3}
%    Is there a sufficiently large $n$ such that for all $s \in \N$,  $C_{2s+1}^{\square n} \xrightarrow{2} C_6$?
%\end{Question}

Recall that for $r$-colorings of Cartesian powers of cliques $K_s$, the case $r=s/2$ represents a natural threshold: if $r>s/2$ then \Cref{prop:necessary} implies that $K_{s}^{\square n} \xrightarrow{r} H$ may hold only for layered graphs, while if $r<s/2$, then $K_{s}^{\square n} \xrightarrow{r} K_3$ already for $n=1$ by the pigeonhole principle. As for the case $r=s/2$, \Cref{2r=s} implies that if $K_{2r}^{\square n} \xrightarrow{r} H$ for some $n \in \N$, then  $H$ is cubical, $C_4$-free, and $\Theta_{3,3,3}$-free.  Moreover, ChatGPT-6 Astra found a non-layered graph $H$ such that $K_{4}^{\square n} \xrightarrow{2}H$ for some $n \in \N$, and we outline the argument below.

An exhaustive search shows that, up to natural symmetries, there exists a unique 2-coloring of $K_4^{\square 4}$ that contains no monochromatic copies of $C_3$ and $C_4$. In particular, we conclude that $K_4^{\square 4} \xrightarrow{2} \{C_3,C_4,B\}$, where  $B$ is a certain 3-regular non-layered cubical graph $B$ on $32$ vertices. Then Lemma~2 from~\cite{MS} allows us to remove a triangle from this family, i.e., show that $K_4^{\square n} \xrightarrow{2} \{C_4,B\}$ for some $n \in \N$. Now \Cref{grapes} shows that $K_4^{\square n} \xrightarrow{2} H$ whenever $H$ is a subgraph of a graph that can be obtained by attaching copies of $C_4$ or $B$ to a middle layer graph in a certain way. ChatGPT-6 Astra found an explicit such non-layered graph $H$ on 18 vertices with 21 edges, which is a subdivision of $K_{3,3}$, whose subdivision paths have $1,1,4,2,2,3,3,3,2$ edges, respectively.

This proof is ad hoc and does not seem to generalize to larger values of $r$. However, we still suspect that, for each $r\ge 2$, there exists a non-layered graph $H$ such that $K_{2r}^{\square n} \xrightarrow{r}H$ for some $n \in \N$. It would also be interesting to find a characterization of such graphs at least for $r=2$.

\begin{Question}
    Given $r >2$, is there some non-layered $H$ and some $n \in \N$ such that $ K_{2r}^{\square n} \xrightarrow{r} H$?
\end{Question}

\begin{Question}
    For which graphs $H$ is there some $n \in \N$ such that $K_4^{\square n} \xrightarrow{2} H$?
\end{Question}

The regime $ \chi(G)>2r$ is substantially different since, in addition to layered graphs, sufficiently long odd cycles also become unavoidable in $r$-colored Cartesian powers of $G$. More specifically, \Cref{arrow_cycles} states that for any graph $G$ and positive integers $r,\ell$ such that $G \xrightarrow{r}\{C_3, C_5, \dots, C_{2\ell+1}\}$ and for any $s\ge\ell+2$, there exists $n\in \N$ such that $G^{\square n} \xrightarrow{r} C_{2s+1}$. 
 We do not know if the inequality can be further relaxed to $s \ge \ell+1$, or if the statement of \Cref{arrow_cycles} is already best possible. We illustrate this problem in the simplest case $\ell=1$.

\begin{Question}
    Is there a graph $G$ such that $G \xrightarrow{2} C_3$ but $G^{\square n} \not \xrightarrow{2} C_5$ for any $n \in \N$?
\end{Question}

We claim that both \Cref{arrow_cycles} and \Cref{grapes} admit natural induced analogues, and the proofs can be repeated verbatim in this case. However, the odd cycles we construct in the proofs of \Cref{arrow_cycles} are never induced, raising the following question.

\begin{Question}
    Suppose that $G \xrightarrow{2} C_3$. Are there $s\ge 2$ and $n \in \N$ such that for every $2$-coloring of the vertices of $G^{\square n}$, there is a monochromatic  induced copy of $C_{2s+1}$ in $G^{\square n}$?
\end{Question}

Our main technical Theorem \ref{grapes} gives in particular a large class $\cF_1$ of such graphs $H$ that are obtained by attaching copies of some graph $F$ with nice Ramsey properties to special edges of a layered graph.
A result in \cite{ALS} also gives a class $\cF_2$ of such graphs $H$ that are obtained by attaching a copy of $F$ to each edge of a graph of zero Tur\'an density in a hypercube.  Note that these classes are distinct,  $\cF_1 \setminus \cF_2 \neq \varnothing$ and 
$\cF_2 \setminus \cF_1 \neq \varnothing$.  For example, when $r=2$, $G=K_5$, and  $H$ is equal to $C_6$ with three attached triangles, we have that $H \in \cF_1\setminus \cF_2$, since $C_6$ has a positive Tur\'an density in a hypercube and $H$ cannot be obtained from another layered graph by attaching graphs arrowed by $K_5$. On the other hand, $C_8$  with triangles attached to all edges is in $\cF_2\setminus \cF_1$ since $C_8$ has zero Tur\'an density in a hypercube.  Here, we recall that all even cycles except for $C_4, C_6, $ and $C_{10}$ have zero Tur\'an density in a hypercube, see \cite{Ch, Conder, Conlon, FO, GM}. We do not know if there is a common strengthening of these results, for example whether  $K_5^{\square n} \xrightarrow{2} \widehat{C}_6$, where $\widehat{C}_6$ is a graph obtained from $C_6$ by attaching triangles to all 6 edges.

 We remark that Ramsey properties for vertex colorings of general graphs have been studied, see for example \cite{Folkman1970, Dudek2012, DudekRodl2010, XuLiangRadziszowski2020, NesetrilRodl1976, KurekRucinski1994, LRV}. Edge-Ramsey properties of Cartesian powers of graphs were considered in \cite{GKS}. The vertex-coloring setting in the current paper has connections to edge colorings of hypergraphs since, for example, the vertices of $Q_n$ in layer $r$ correspond to hyperedges of a complete $r$-uniform hypergraph on the vertex set $[n]$.

\noindent
{\bf AI disclosure.~~} ChatGPT-5.6 Plus assisted with proofreading the manuscript and creating the tikz code for the figure. ChatGPT-6 Astra found two arguments briefly outlined in \Cref{conclusions}. The rest is due to the authors. 

\noindent
{\bf Acknowledgements.~~} The first author's work is supported by the Doctoral Excellence Fellowship Programme (DCEP) and funded by the Budapest University of Technology and Economics and the National Research, Development and Innovation Office (NKFIH), ADVANCED-152590.

\end{document}